\documentclass[10pt, reqno]{amsart}
\usepackage{eurosym}
\usepackage{amsmath, amsthm, amscd, amsfonts, amssymb, graphicx, color}
\usepackage{xcolor}

\newtheorem{theorem}{Theorem}[section]
\newtheorem{lemma}[theorem]{Lemma}

\newtheorem{corollary}[theorem]{Corollary}
\theoremstyle{definition}

\newtheorem{remark}{Remark}
\numberwithin{equation}{section}
\usepackage{pdflscape}
\usepackage{longtable}
\usepackage{xcolor}
\usepackage{amsmath}
\usepackage[
colorlinks=true,
citecolor=red,
linkcolor=blue,
urlcolor=blue
]{hyperref}

\begin{document}
\title[Coefficient Difference]{On the Moduli Difference of Initial Logarithmic and Inverse Logarithmic Coefficients of Ma-Minda Convex Functions} 
\author[ R. Ali, M. Raza, U. Raza]{Rashid Ali$^{1}$, Mohsan Raza$^{1,\ast}$, Umar Raza$^{2}$}

\address{$^{1}$Department of Mathematics, Government College University Faisalabad, Pakistan.}
\email{rashidali5277@gmail.com}
\email{mohsan976@yahoo.com}

\address{$^{2}$Department of Mathematics, University of Jhang, Pakistan.}
\email{umarraza55@gmail.com}


\date{Received \today.\\
	\indent$^{\ast}$ Corresponding author\\
	2020\textit{\ Mathematics Subject Classification. }30C45, 30C50.}

\begin{abstract}
    Let  $\mathcal{C}(\Phi)$ denote a subclass of class $
\mathcal{C}$ of convex functions in the unit disc $\mathbb{D}=\{z:|z|<1,
z\in \mathbb{C}\}$ defined by,

\begin{equation*}
{Re}\left \{1+ \frac{
zf^{\prime \prime}(z)}{f^{\prime}(z)}\right \} \prec \Phi(z).
\end{equation*}
where $\Phi$ represents the Ma--Minda function in $\mathbb{D}$. For $f\in\mathcal{C}(\Phi)$,
	we find the sharp estimates of $||\gamma_{2}|-|\gamma_{1}||$, where $\gamma_{1}$
	and $\gamma_{2}$ are the first two logarithmic coefficients of $f$. The bound
	is also obtained for the inverse logarithmic coefficients
	$\Gamma_{1}$ and $\Gamma_{2}$. We then show that, under suitable conditions,
	the bounds for $||\gamma_{2}|-|\gamma_{1}||$ and $||\Gamma_{2}|-|\Gamma_{1}||$
	are identical, which gives the invariance result between these functionals. We also present bounds for the differences for various classes of univalent functions by choosing particular function $\Phi$. Thus our results unify several results studied in the literature.
\end{abstract}

\maketitle


\setcounter{page}{1} 

\section{Introduction}

Let $\mathcal{A}$ be the set of functions holomorphic in the unit disk
$\mathbb{D}=\{z\in\mathbb{C}:|z|<1\}$ that satisfy $f(0)=0=f'(0)-1$. Every
$f\in\mathcal{A}$ therefore has the series form
\begin{equation}
f(z)=z+\sum_{n=2}^{\infty}a_{n}z^{n},\qquad z\in\mathbb{D}. \label{1.1}
\end{equation}
Within class $\mathcal{A}$, the univalent (one-to-one) functions form the class $\mathcal{S}$. For this class, de Branges \cite{c2} established in 1985 that $|a_n|\leq n,\qquad n\geq2$, which is known as the Bieberbach conjecture. Equality holds for the well-known Koebe function $k(z)=\frac{z}{(1-z)^2}$ and its rotations. Following this result, it is natural to compare consecutive coefficients, leading to the conjecture
\begin{equation*}
\bigl||a_{n+1}|-|a_n|\bigr|\leq1,\qquad n\geq2.
\end{equation*}
However, this conjecture is false for the class $\mathcal{S}$. Indeed, a counterexample already occurs for $n=2$. In fact, Duren \cite{c1} showed that
\begin{equation*}
-1\leq |a_3|-|a_2|
\leq \frac{3}{4}+e^{-\eta_0}\left(2e^{-\eta_0}-1\right)
=1.029\cdots,
\end{equation*}
where $\eta_0$ denotes the unique root of the equation $4\eta=e^\eta$ in the interval $(0,1)$.\\
Hayman \cite{c3} proved that the coefficient differences are uniformly
bounded in general, i.e.\ $\big||a_{n+1}|-|a_{n}|\big|\le C$ valid for all $f\in\mathcal{S}$ and $n\ge2$, where
$C$ is an absolute constant. In the original proof of the result, Hayman \cite{c5} employed a powerful method that he had developed for studying mean $p$-valent functions. Milin, later reached
the same conclusion by using the Lebedev-Milin inequalities. Duren's classic monograph \cite{c1} provides a thorough explanation of these developments. The value of $C$ has proved hard to lower. Later, Ilina \cite{lina}
showed that $C<4.26\cdots$ in 1968. Eight years later, the author refined Milin's argument
 to obtain
\begin{equation*}
-2.97\cdots<|a_{n+1}|-|a_{n}|<3.61\cdots ,
\end{equation*}
for $n\ge 2$, see \cite{c4}. No more advancements were reported for several years until Obradovi\'{c} et al.\ \cite{ob} brought in the
Grunsky inequalities and found $|a_{4}|-|a_{3}|\le 2.1033\cdots$. Accordingly,
beyond $n=2$ nobody knows the exact extreme values of $|a_{n+1}|-|a_{n}|$ for
the class $\mathcal{S}$.

A number of authors have investigated coefficient problems for various subclasses of the class $\mathcal{S}$. We begin by recalling some of the most important and well-known subclasses of $\mathcal{S}$. A function $f\in\mathcal{S}$ is said to be \emph{starlike} if $f(\mathbb{D})$ is a domain starlike with respect to the origin. The class of all such functions is denoted by $\mathcal{S}^{\ast}$. Analytically, $f\in\mathcal{S}^{\ast}$ if and only if
\begin{equation*}
\Re\left(\frac{zf'(z)}{f(z)}\right)>0,\qquad z\in\mathbb{D}.
\end{equation*}

Similarly, a function $f\in\mathcal{S}$ is called \emph{convex} if $f(\mathbb{D})$ is a convex domain. The class of all convex functions is denoted by $\mathcal{C}$ and is characterized analytically by
\begin{equation*}
\Re\left(1+\frac{zf''(z)}{f'(z)}\right)>0,\qquad z\in\mathbb{D}.
\end{equation*}

For starlike functions, Pommerenke \cite{c5} proposed in 1973 that
$\bigl||a_{n+1}|-|a_{n}|\bigr|\le 1$ for $n\ge2$; Leung \cite{leu} verified
this in 1978 and showed that the result is sharp for the function
\begin{equation*}
f(z)=\frac{z}{(1-\rho z)(1-\sigma z)},\qquad |\rho|=|\sigma|=1.
\end{equation*}
For the class of close-to-convex functions, Koepf \cite{koepf} obtained that $\bigl||a_{n+1}|-|a_{n}|\bigr|\le 1$ for $n=2$, whereas the corresponding result for $n\geq3$ remains unproved. For convex functions the
sharp upper and lower bounds for  $|a_{n+1}|-|a_{n}|$ are still challenging to find. The only significant findings to date are credited to  Ming and Sugawa \cite{Ming} who found the
best possible upper bounds for $n\ge2$ and the best possible lower bounds for $n=2,3$,
leaving the lower bound for $n\ge4$ as an open problem. Analogous questions about the
initial coefficients have been addressed by many authors for different classes:
Cho et al. \cite{cho} established bounds for $||a_3|-|a_2||$ for the class of Bazilevič functions. More recently, the corresponding bounds for the classes of non-Bazilevič and $\mathcal{U}(\alpha,\lambda)$ were obtained in \cite{ali,Raza} respectively.

Let $\mathcal{B}$ denote the class of Schwarz functions, i.e.\ $w:\mathbb{D}\to\mathbb{D}$
with $w(0)=0$. A function $f$ is subordinate to $g$ written mathematically as $f\prec g$ if there exists a function $w\in\mathcal{B}$ such that 
$f=g\circ w$. If $g$ is univalent and
$f(0)=g(0)$, then $f(\mathbb{D})\subset g(\mathbb{D})$.

Ma and Minda \cite{11} used subordination to establish a number of subclases of 
classical starlike and convex classes. They take a univalent $\Phi$ on
$\mathbb{D}$ whose image is symmetric about the real axis and starlike with
respect to $\Phi(0)=1$, with $\Phi'(0)>0$, and define $\mathcal{C}(\Phi)$ as
the class of functions $f\in\mathcal{A}$ for which
\begin{equation}\label{1.2}
1+\frac{zf''(z)}{f'(z)}\prec\Phi(z),\qquad z\in\mathbb{D}.
\end{equation}
Particular choices of $\Phi$ give familiar subclasses of $\mathcal{C}$,
including the following.

\begin{enumerate}
	\item[i.] $\mathcal{C}_{\alpha}:=\mathcal{C}\left(\left(\frac{1+z}{1-z}\right)^{\alpha}\right)$, see \cite{stan}.
	\item[ii.] $\mathcal{C}_{SG}:=\mathcal{C}(\frac{2}{1+e^{-z}})$, see \cite{goel,raza4}.
	\item[iii.] $\mathcal{C}_{q_{c}}:=\mathcal{C}(\sqrt{1+cz})$, $0<c\le1$, see \cite{sokol1}.
	\item[iv.] $\mathcal{BS}(\alpha):=\mathcal{C}(1+\frac{z}{1-\alpha z^{2}})$, $0\le\alpha\le1$, see \cite{kar}.
	\item[v.] $\mathcal{C}_{l}:=\mathcal{C}(\sqrt{1+z^{2}}+z)$, see \cite{raina,raza2}.
	\item[vi.] $\mathcal{C}_{c}:=\mathcal{C}\left(1+\frac{4z}{3}+\frac{2z^{2}}{3}\right)$, see \cite{raza3,sharma}.
	\item[vii.] $\mathcal{C}_{RL}:=\mathcal{C}\left(\sqrt{2}-(\sqrt{2}-1)\sqrt{\frac{1-z}{1+2(\sqrt{2}-1)z}}\right)$, see \cite{mendi}.
	\item[viii.] $\mathcal{C_L}:=\mathcal{C}(\sqrt{1+z})$, see \cite{sokol}.
	\item[ix.] $\mathcal{C}_{s}:=\mathcal{C}(1+\sin z)$, see \cite{cho1}.
	\item[x.] $\mathcal{C}_{e}:=\mathcal{C}(e^{z})$, see \cite{raza1,mend}.
	\item[xi.] $\mathcal{C}[A,B]:=\mathcal{C}\left(\frac{1+Az}{1+Bz}\right)$, $-1\le B<A\le1$, see \cite{janowski}.
\end{enumerate}

Consider the function $\Phi(z)$ whose power series expansion is given by
\[
\Phi(z)=1+D_{1}z+D_{2}z^{2}+D_{3}z^{3}+\cdots.
\] 
The well-known Koebe's one quarter theorem ensures that inverse of each $f\in\mathcal{S}$ is 
defined at least in $|z|<1/4$, with series expansion given as
\begin{equation}
f^{-1}(w)=z+\sum_{n=2}^{\infty }A_{n}w^{n}.  \label{1.3}
\end{equation}
Substituting (\ref{1.3}) into $f(f^{-1}(w))=w$ and using (\ref{1.1}) gives
\begin{equation}
A_{2}=-a_{2},\ \ \ A_{3}=2a_{2}^{2}-a_{3}.  \label{1.4}
\end{equation}
Libera and Zlotkiewicz \cite{lib1} proved that for convex function $f$, the inverse
coefficients satisfy the classical estimate $|A_{n}|\le1$ for $2\le n\le7$.
A phenomenon of this kind, in which a bound associated with the function $f$ and the corresponding bound for $f^{-1}$ are same, is referred to as invariance. Invariance of bounds for certain coefficient functionals was established for strongly convex functions by Thomas and Verma \cite{DK}, and a similar property was obtained for convex functions subordinate to the sigmoid function in \cite{raza4}. Recently, Thomas provided a comprehensive analysis of this characteristic for various coefficient functionals arising in the class $\mathcal{C}$ and its subclasses, (see \cite{raza3}).

Bounds for coefficients of inverse functions have
also been studied widely (see \cite{m30,m37}). Particularly, Sim and Thomas \cite{m30}
showed that $-1\le|A_{3}|-|A_{2}|\le3$ for the class $\mathcal{S}$ and also obtained sharp
estimates on several other classes.

The logarithmic coefficients $\gamma_{n}$ for $f\in\mathcal{S}$ are defined as
\begin{equation}
\frac{F_{f}(z)}{2}=\frac{1}{2}\log \frac{f(z)}{z}=\sum_{n=1}^{\infty }\gamma_{n}z^{n},\ \ \ z\in \mathbb{D},  \label{equ24}
\end{equation}
and comparison with (\ref{1.1}) gives
\begin{equation}
\gamma _{1}=\frac{a_{2}}{2},\ \ \ \gamma _{2}=\frac{1}{4}\left( 2a_{3}-a_{2}^{2}\right) .  \label{1.6}
\end{equation}
In function theory, logarithmic coefficients \(\gamma_n\) are of particular importance nd have been extensively studied. Milin \cite{m18} drew the link to the Bieberbach conjecture by
proposed that for $f\in\mathcal{S}$ and $n\ge2$,
\[
\sum_{m=1}^{n}\sum_{k=1}^{m}\left(k|\gamma_{k}|^2 - \frac{1}{k}\right) \leq 0,
\]
which de Branges proved and then used to settle the Bieberbach conjecture
\cite{c2}. For the Koebe function $k(z)=z/(1-z)^2$ one gets $\gamma_n=1/n$ but $|\gamma_n|\le 1/n$ is false in general for $f\in\mathcal{S}$, which explains the continuing interest of authors in logarithmic coefficient for the class $\mathcal{S}$ and its
subclasses (see \cite{m1,raza1,m36}).

Applying Loewner theory, Lecko and
Partyka \cite{m14} showed in 2023 that for $f\in \mathcal{S}$
\begin{equation*}
-\frac{\sqrt{2}}{2} \leq |\gamma_{2}| - |\gamma_{1}| \leq \frac{1}{2},
\end{equation*}
the right hand side inequality being sharp for $f(z)=z/(1+\tau^{2}z^{2})$ and the left
for $f(z)=z/(1-\sqrt{2}\tau z+\tau^{2}z^{2})$, where $|\tau|=1$. Furthermore, Obradovic \cite{nikola} studied bounds on  $|\gamma_{3}|-|\gamma_{2}|$ and
$|\gamma_{4}|-|\gamma_{3}|$ for the class $\mathcal{S}$. Similarly, Obradovic \cite{nikola25} presented sharp
bounds on $|\gamma_{2}|-|\gamma_{1}|$ for the class $\mathcal{B}_{1}(\alpha)$ of Bazilevic functions. Ali and Raza
\cite{ali} established the same results for the class of non-Bazilevi\v{c} functions and Kumar and Cho
\cite{kumarcho} provided sharp two-sided bounds on $|\gamma_{2}|-|\gamma_{1}|$ for
certain subclasses of $\mathcal{S}$.

The inverse logarithmic coefficients $\Gamma_{n}$ for $f\in\mathcal{S}$ is presented by Ponnusamy et al.\ \cite{ponusamy18} and given as
\begin{equation}
\frac{F_{f^{-1}}(z)}{2}=\frac{1}{2}\log \frac{f^{-1}(z)}{z}=\sum_{n=1}^{\infty }\Gamma _{n}z^{n},\ \ \ z\in \mathbb{D}.  \label{equ25}
\end{equation}
Using(\ref{1.4}), we get
\begin{equation}
\Gamma _{1}=-\frac{a_{2}}{2},\ \ \ \Gamma _{2}=\frac{1}{4}(-2a_{3}+3a_{2}^{2}).  \label{Gamma1Gamma2}
\end{equation}
Allu and Shaji \cite{allu} showed that
\begin{equation}
\Delta \leq |\Gamma _{2}|-|\Gamma _{1}|\leq \frac{1}{2}  \label{equdel}
\end{equation}
for $f\in\mathcal{S}$, where $\Delta=-\tfrac12$ if $a_{2}\le1$ and
$\Delta=-0.6353\ldots$ if $a_{2}>1$. Equality on the right is attained by
$f(z)=z/(1+z+z^{2})$, whether establishing the extremal function for left inequality is still open. The same
paper also established bounds $||\Gamma_{2}|-|\Gamma_{1}||$ for several subclasses. Recently, 
\cite{ali} gives sharp bounds on this difference for non-Bazilevi\v{c}
functions.

Invariance of bounds on moduli differences has been examined for coefficients
and for inverse coefficients over many subclasses of $\mathcal{A}$ but the
corresponding question for logarithmic and inverse logarithmic coefficients,
which are more intricate, has not been taken up. We fill this gap for the
Ma-Minda convex class $\mathcal{C}(\Phi)$. We establish sharp bounds for
$||\gamma_{2}|-|\gamma_{1}||$ and $||\Gamma_{2}|-|\Gamma_{1}||$ and identify when
they are invariant. We also give examples of both invariant and non-invariant
cases.

\section{Main Results}
Let $\mathcal{P}$ be the Carath\'eodory class of functions $h$ holomorphic in
$\mathbb{D}$ with $\Re (h(z))>0$ and
\begin{equation*}
p(z)=1+\sum_{n=1}^{\infty }p_{n}z^{n}.
\end{equation*}
The proof of our results relies on the following lemmas.
\begin{lemma}
\label{lemma1}\cite{19} Let $h\in \mathcal{P}$. Then 
\begin{equation*}
p_{1}=2t_{1},  
\end{equation*}%
and 
\begin{equation*}
p_{2}=2t_{1}^{2}+2(1-|t_{1}|^{2})t_{2}, 
\end{equation*}%
where $t_{i}\in \overline{\mathbb{D}}=\{z:|z|\le1,
z\in \mathbb{C}\}$ for $i\in \left\{ 1,2\right\} $. Now
if $|t_{1}|=1$, there is a unique function $h\in \mathcal{P}$ with $p_1$ given above is provided
as  
\begin{equation*}
h(z)=\frac{1+t_{1}z}{1-t_{1}z}.
\end{equation*}%
If $t_{1}\in \mathbb{D}$ and $|t_{2}|=1$, there is a unique function $h\in \mathcal{P}$ with $p_2$ as given above is provided
as  
\begin{equation}
h(z)=\frac{1+(\overline{t_{1}}t_{2}+t_{1})z+t_{2}z^{2}}{1+(\overline{t_{1}}%
t_{2}-t_{1})z-t_{2}z^{2}}.  \label{2.1}
\end{equation}
\end{lemma}
Our first result, develop a sharp bound on the functional  $||\gamma_{2}|-|\gamma_{1}||$.
\begin{theorem}\label{Theorem2.2}
If $f\in\mathcal{C}(\Phi)$ is of the form \eqref{1.1} and $D_1>0$, then 
\begin{equation*}
|\gamma_{2}|-|\gamma_{1}| \leq 
\begin{cases}
	\frac{\left\vert D_{1}^2+4D_{2}\right\vert-12D_{1} }{48}, & \text{if }\ |D_{1}^2+4D_{2}|\geq 16 D_{1}, \\ 
	\\
	\frac{D_{1}}{12}, & \text{if }\ |D_{1}^2+4D_{2}|< 16 D_{1},%
\end{cases}%
\end{equation*}%
\newline
\newline
and 
\begin{equation*}
|\gamma_{2}|-|\gamma_{1}| \geq 
\begin{cases}
	\frac{\left\vert D_{1}^2+4D_{2}\right\vert-12D_{1} }{48}, & \text{if }\ 2D_{1}\geq | D_{1}^{2}+4D_{2}|, \\[8pt]
	-\frac{D_{1}}{2}\sqrt{\frac{D_{1}}{4D_{1}+|D_{1}^2+4D_{2}|}}, & \text{if }\ 5D_{1}\leq | D_{1}^{2}+4 D_{2}|, \\%
	[8pt]
	-\frac{ D_{1}}{12}-\frac{ {3 D_{1}}^{2}}{4 {| D_{1}^{2}+4D_{2}|}+16 { D_{1}}}, & \text{otherwise}.%
\end{cases}%
\end{equation*}%
These bounds are sharp.
\end{theorem}

\begin{proof}
	Since \(f\in\mathcal{C}(\Phi)\), therefore their exists a function $w \in \mathcal{B}$ satisfying
	\begin{equation}
	1+\frac{zf^{\prime\prime }(z)}{f^{\prime}(z)}=\Phi \left( w(z)\right) .  \label{2.2}
	\end{equation}%
	Let $h\in \mathcal{P}$. Then 
	\begin{equation}
	h(z)=\frac{1+w(z)}{1-w(z)}=1+p_{1}z+p_{2}z^{2}+p_{3}z^{3}+\cdots .  \label{2.3}
	\end{equation}%
    From (\ref{2.2}) and (\ref{2.3}), we write
    \begin{equation}
	1+\frac{zf^{\prime\prime }(z)}{f^{\prime}(z)}=\Phi \left( \frac{h(z)-1}{h(z)+1}\right) .  \label{2.2a}
	\end{equation}%
	Now, by comparing the coefficients of \eqref{2.2} and \eqref{2.3}, we have 
	\begin{equation}
	a_{2}=\frac{D_{1}}{4}p_{1},\quad \quad a_{3}:=\left( \frac{%
		-D_{1}+D_{1}^2+D_{2}}{24}\right) p_{1}^{2}+\frac{D_{1}}{12}p_{2}.  \label{2.4}
	\end{equation}%
	Therefore from \eqref{1.6}, we have
	\begin{equation}
		\gamma_{1}=\frac{D_{1}}{8}p_{1},\quad \quad \gamma_{2}:=\left( \frac{%
			-4D_{1}+D_{1}^2+4D_{2}}{192}\right) p_{1}^{2}+\frac{D_{1}}{24}p_{2}.  \label{2.5}
	\end{equation}%
	By using \eqref{2.5}, we get 
	\begin{equation*} 
	|\gamma_{2}|-|\gamma_{1}|=|{u}p_{1}^{2}+{v}p_{2}|-|{s}p_{1}|,
	\end{equation*}%
	with 
	\begin{equation*}
	s:=\frac{D_{1}}{8},\quad \quad u:=\left( \frac{-4D_{1}+D_{1}^2+4D_{2}}{192}%
	\right) ,\quad \text{and}\quad v:=\frac{D_{1}}{24}.
	\end{equation*}%
By using Lemma \ref{lemma1}, we get 
\begin{align*}
|\gamma_{2}|-|\gamma_{1}| & ={|4ut_{1}^{2}+v(2t_{1}^{2}+2(1-{|t_{1}|}^{2})t_{2})|}-2{|st_{1}|} \\
& \leq {j|t_{1}|}^{2}+2{|v|}\left( 1-{|t_{1}|}^{2}\right) |t_{2}|-2{|st_{1}|},
\end{align*}%
where $j:=|4u+2v|$.\\
By using the fact that the class $\mathcal{P}$ is rotationally invariant, we may assume that $t_{1}\in
\lbrack 0,1]$ and by using the fact that $|t_{2}|\leq 1$, we have 
\begin{align*}
|\gamma_{2}|-|\gamma_{1}| & \leq \left( {j}-2{v}\right) t_{1}^{2}-2{s}t_{1}+2{v}=\Psi (x),
\end{align*}%
where 
\begin{equation}
\Psi (x)=k_{2}x^{2}+k_{1}x+k_{0},  \label{2.6}
\end{equation}%
with 
\begin{equation*}
k_{2}={j}-2{v},\quad \quad k_{1}=-2{s},\quad \quad k_{0}=2{v}.
\end{equation*}%
Here we have two cases:\newline
$\mathbf{A}_{1}:$ If $k_{2}\leq 0$. Given that $k_{1}<0$, thus by (\ref{2.6}
) we get, $\Psi ^{\prime }(x)=2k_{2}x+k_{1}<0$ and it ensures that the function $
\Psi $ is decreasing. Therefore 
\begin{equation*}
\Psi (x)\leq \Psi (0)=k_{0}=2{v}=\frac{D_{1}}{12}.
\end{equation*}%
$\mathbf{A}_{2}:$ If \(k_{2}>0\), then the function \(\Psi\) is a quadratic and has leading coefficient positive. Hence, it is convex and attains its minimum value at the vertex, therefore 
\begin{equation*}
\Psi (x)\leq \max \{\Psi (0),\Psi (1)\}.
\end{equation*}%
Let $k_{2}+k_{1}<0$. Then $\Psi (1)=k_{2}+k_{1}+k_{0}<k_{0}=\Psi (0)$
and thus $\Psi (x)\leq \Psi (0)=2{v}$. Further, let $%
k_{2}+k_{1}\geq 0$. Then $\Psi (1)=k_{2}+k_{1}+k_{0}\geq k_{0}=\Psi
(0)$. Thus 
\begin{equation*}
\Psi (x)\leq \Psi (1)=k_{2}+k_{1}+k_{0}=|4{u}+2{v}|-2{s}=\frac{\left\vert D_{1}^2+4D_{2}\right\vert-12D_{1} }{48}.
\end{equation*}%
Thus, from above discussion, we have 
\begin{equation*}
	|\gamma_{2}|-|\gamma_{1}| \leq 
	\begin{cases}
		\frac{\left\vert D_{1}^2+4D_{2}\right\vert-12D_{1} }{48}, & \text{if }\ |D_{1}^2+4D_{2}|\geq 16 D_{1}, \\ 
		\\
		\frac{D_{1}}{12}, & \text{if }\ |D_{1}^2+4D_{2}|< 16 D_{1}.%
	\end{cases}%
\end{equation*}%
For $|D_{1}^2+4D_{2}|\geq 16 D_{1}$, the function $f$ given by \eqref{2.2a} provides a sharp result with $h(z)=(1+z)/(1-z)$.  Also, it is best possible for the function $f$ given by \eqref{2.2a} with $h(z)=(1+z^{2})/(1-z^{2})$ when $|D_{1}^2+4D_{2}|< 16 D_{1}$.\newline
Now, we consider the lower bound on $|\gamma_{2}|-|\gamma_{1}| $.\\
By using Lemma \ref{lemma1}, we get 
\begin{equation*}
|\gamma_{2}|-|\gamma_{1}| =\left\vert {j}e^{i\theta }t_{1}^{2}+2{v}e^{i\phi
_{1}}(1-|t_{1}|^{2})t_{2}\right\vert -2{s}t_{1},
\end{equation*}%
where 
\begin{equation*}
{j}=\left\vert 4{u}+2{v}\right\vert ,\quad \quad \theta =arg(4{u}+2{v}),\quad \quad
\phi _{1}=arg(2{v}).
\end{equation*}%
Since $\mathcal{P}$ is rotationally invariant, we may assume that \(t_{1}\in[0,1]\) and \(t_{2}=re^{i\phi_{2}}\), where \(r\in[0,1]\). Thus 
\begin{equation*}
|\gamma_{2}|-|\gamma_{1}| =\left\vert {j}e^{i\theta }t_{1}^{2}+2r{v}e^{i\phi
}(1-t_{1}^{2})\right\vert -2{s}t_{1},
\end{equation*}%
with $\phi =\phi _{1}+\phi _{2}$. As $\left\vert e^{i\phi }\right\vert =1$,
so 
\begin{equation*}
|\gamma_{2}|-|\gamma_{1}| =\left\vert {j}e^{i(\theta -\phi
)}t_{1}^{2}+2r{v}(1-t_{1}^{2})\right\vert -2{s}t_{1}.
\end{equation*}%
Now 
\begin{align}
|\gamma_{2}|-|\gamma_{1}| =& \sqrt{{j}^{2}t_{1}^{4}+4{j}{v}rt_{1}^{2}(1-t_{1}^{2})\cos (\theta -\phi
)+4{v}^{2}r^{2}(1-t_{1}^{2})^{2}}-2{s}t_{1},  \notag \\
\geq & \left\vert {j}t_{1}^{2}-2r{v}(1-t_{1}^{2})\right\vert -2{s}t_{1},
\label{2.7}
\end{align}%
as $\cos (\theta -\phi )\geq -1$. If ${u}={v}=0$, then 
\begin{equation*}
|\gamma_{2}|-|\gamma_{1}| \geq -2{s}t_{1}\geq -2{s}=-\frac{D_{1}}{4}.
\end{equation*}%
Now, let ${u}\neq 0,\ {v}\neq 0$. Then from \eqref{2.7}, we have two cases:%
\newline
\textbf{B}$_{1}$: If ${j}t_{1}^{2}-2r{v}(1-t_{1}^{2})\leq 0$, then 
\begin{equation*}
t_{1}\leq \sqrt{\frac{2r{v}}{{j}+2r{v}}}:=\zeta _{1}.
\end{equation*}%
So, from (\ref{2.7}) and by using the fact, $r\in \lbrack 0,1]$, we have 
\begin{align*}
|\gamma_{2}|-|\gamma_{1}| & \geq -({j}+2r{v})t_{1}^{2}-2{s}t_{1}+2r{v} \\
& \geq -({j}+2r{v})\zeta _{1}^{2}-2{s}\zeta _{1}+2r{v} \\
& =-2{s}\sqrt{\frac{2r{v}}{{j}+2r{v}}}=-2{s}\sqrt{\frac{r}{\frac{j}{2{v}}+r}}
\\
& \geq -2{s}\sqrt{\frac{1}{\frac{j}{2{v}}+1}} \\
& ={-2{s}}\sqrt{\frac{2{v}}{{j}+2{v}}}\\
& =-\frac{D_{1}}{2}\sqrt{\frac{D_{1}}{4D_{1}+|D_{1}^2+4D_{2}|}}.
\end{align*}%
\textbf{B}$_{2}$: If ${j}t_{1}^{2}-2r{v}(1-t_{1}^{2})>0$, then $t_{1}>\zeta
_{1}$, and we define 
\begin{equation*}
\varphi (t_{1},r)=({j}+2r{v})t_{1}^{2}-2{s}t_{1}-2r{v}.
\end{equation*}%
Since, for ${s}^{2}\leq 2{v}({j}+2{v})$ and by using the fact that $r\in
\lbrack 0,1]$, we have 
\begin{equation*}
\frac{\partial \varphi }{\partial t_{1}}=2({j}+2r{v})t_{1}-2{s}\geq 2\sqrt{%
2r{v}({j}+2r{v})}-2{s}\geq 0.
\end{equation*}%
It shows that the function $\varphi$ is increasing, and by using (\ref{2.7}), we have 
\begin{equation*}
|\gamma_{2}|-|\gamma_{1}| \geq -2{s}\sqrt{\frac{2r{v}}{{j}+2r{v}}}\geq -2{s}\sqrt{\frac{2{%
{v}}}{{j}+2{v}}}=-\frac{D_{1}}{2}\sqrt{\frac{D_{1}}{4 D_{1}+|D_{1}^2+4D_{2}|}},
\end{equation*}%
as $r\in \lbrack 0,1]$. Now, for ${s}\geq {j}+2{v}$, we deduce 
\begin{equation*}
\frac{\partial \varphi }{\partial t_{1}}=2({j}+2r{v})t_{1}-2{s}\leq 2{j}+4r{%
{v}}-2{s}\leq 2{j}+4{v}-2{s}\leq 0.
\end{equation*}%
It ensures that the function $\varphi$ is increasing, and by using (\ref{2.7}), we get 
\begin{equation*}
|\gamma_{2}|-|\gamma_{1}|\geq {j}-2{s}=\frac{\left\vert D_{1}^2+4D_{2}\right\vert-12 D_{1} }{48}.
\end{equation*}%
At the end, for ${s}^{2}>2{v}({j}+2{v})$ and ${s}<{j}+2{v}$, we have 
\begin{equation*}
\zeta _{1}<\frac{{s}}{{j}+2{v}}:=\zeta _{2}<1,
\end{equation*}%
and $\frac{\partial \varphi }{\partial r}=-2{v}(1-t_{1}^{2})\leq 0$ for $%
r\in \lbrack 0,1]$ and $t_{1}\in \lbrack \zeta _{1},1]$. Thus 
\begin{equation*}
\varphi (t_{1},r)\geq \varphi (t_{1},1)=({j}+2{v})t_{1}^{2}-2{s}t_{1}-2{{v}%
}=\varphi_1(t_{1}).
\end{equation*}%
Since $\varphi_1^{\prime }(t_{1})=0$ yields $t_{1}=\zeta _{2}$. Thus from (\ref{2.7}%
) 
\begin{equation*}
|\gamma_{2}|-|\gamma_{1}| \geq -2{{v}-\frac{{s}^{2}}{{j}+2{v}}}=-\frac{ D_{1}}{12}-\frac{ 3{ D_{1}}^{2}}{4 {| D_{1}^{2}+4D_{2}|}+16 { D_{1}}}.
\end{equation*}%
From above discussion, we have
\begin{equation*}
	|\gamma_{2}|-|\gamma_{1}| \geq 
	\begin{cases}
		\frac{\left\vert D_{1}^2+4D_{2}\right\vert-12D_{1} }{48}, & \text{if }\ 2D_{1}\geq | D_{1}^{2}+4D_{2}|, \\[8pt]
		-\frac{D_{1}}{2}\sqrt{\frac{D_{1}}{4D_{1}+|D_{1}^2+4D_{2}|}}, & \text{if }\ 5D_{1}\leq | D_{1}^{2}+4 D_{2}|, \\%
		[8pt]
		-\frac{ D_{1}}{12}-\frac{ {3 D_{1}}^{2}}{4 {| D_{1}^{2}+4D_{2}|}+16 { D_{1}}}, & \text{otherwise},%
	\end{cases}%
\end{equation*}%
For $ 2D_{1}\geq | D_{1}^{2}+4D_{2}|$, the result is best possible for $f$ given by \eqref{2.2a} with $
h(z)=(1+z)/(1-z)$.\newline
Let $5D_{1}\leq | D_{1}^{2}+4 D_{2}|$, then by using Lemma \ref{lemma1}, the result is sharp for $f$ given by \eqref{2.2a}, where $h(z)$ is provided in \eqref{2.1} with 
\begin{equation*}
t_{1}=\sqrt{\frac{{ 4D_{1}}}{{| D_{1}^{2}+4D_{2}|}+{4 D_{1}}}},\quad \text{and}\quad t_{2}=%
\begin{cases}
-\frac{ (D_{1}^{2}+4D_{2}) }{| D_{1}^{2}+4D_{2}|}, & \text{if }\  D_{1}^{2}+4D_{2}\neq 0, \\[1.2em]
1, & \text{if }\  D_{1}^{2}+4D_{2}=0.%
\end{cases}%
\end{equation*}%
For $ D_{1}^{2}+4 D_{2}=0$, the calculations are obvious and $|\gamma_{2}|-|\gamma_{1}|=-\frac{D_{1}}{4}$.\newline
Now for $D_{1}^{2}+4D_{2}\neq 0$, 
\begin{align*}
|\gamma_{2}|-|\gamma_{1}| &=\left\vert {u}p_{1}^{2}+{v}p_{{2}}|-|{s}p_{1}\right\vert\\
 &=\frac{1}{48}\left\vert \left(D_{1}^{2}+4D_{2}\right)t_{1}^{2}+4D_{1} (1-t_{1}^{2})t_{2}\right\vert-\frac{1}{4}D_{1}t_{1}
\end{align*}%
and after some simplifications, we get 
\begin{equation*}
|\gamma_{2}|-|\gamma_{1}|=-\frac{D_{1}}{4}\sqrt{\frac{4D_{1}}{4D_{1}+|D_{1}^2+4D_{2}|}}.
\end{equation*}%
Lastly, for $ 5D_{1}> | D_{1}^{2}+4 D_{2}|$ and $2D_{1}< | D_{1}^{2}+4D_{2}|$, the result is best possible
for the function $f$ given by \eqref{2.2a}, where $h(z)$ is provided in \eqref{2.1} with 
\begin{equation*}
t_{1}=\frac{6 { D_{1}}}{{| D_{1}^{2}+4D_{2}|}+4 { D_{1}}},\quad \text{and}\quad t_{2}=%
\begin{cases}
-\frac{(D_{1}^{2}+4D_{2}) }{| D_{1}^{2}+4D_{2}|}, & \text{if }\ D_{1}\neq 0, \\[1.2em]
1, & \text{if }\ D_{1}=0.%
\end{cases}%
\end{equation*}%
Now
\begin{align*}
2304|{u}p_{1}^{2}+{v}p_{2}|^{2}& =\left\vert \left(D_{1}^{2}+4D_{2}\right)t_{1}^{2}+4D_{1} (1-t_{1}^{2})t_{2}\right\vert ^{2}\\
& =\left|D_{1}^{2}+4D_{2}\right|^{2}t_{1}^{4}+2Re\left( 4D_{1}(\overline{D_{1}^{2}+4D_{2}})t_{1}^{2}(1-t_{1}^{2})t_{2}%
\right) +16D_{1}^{2}(1-t_{1}^{2})^{2} \\
& =\left|D_{1}^{2}+4D_{2}\right|^{2}t_{1}^{4}-2\left(4 D_{1} |D_{1}^{2}+4 D_{2}| t_{1}^2 (1-t_{1}^{2})\right)+16D_{1}^{2}(1-t_{1}^{2})^{2} \\
& =\left( \left|D_{1}^{2}+4D_{2}\right|t_{1}^{2}-4D_{1}(1-t_{1}^{2})\right) ^{2} \\
& =\left(\frac{36 { D_{1}}^{2}}{ {| D_{1}^{2}+4D_{2}|}+4 { D_{1}}}-4D_{1}\right) ^{2}.
\end{align*}%
Since $5D_{1}> | D_{1}^{2}+4 D_{2}|$, it implies $\frac{36 { D_{1}}^{2}}{ {| D_{1}^{2}+4D_{2}|}+4 { D_{1}}}-4D_{1}>0$ and
thus 
\begin{equation*}
|{u}p_{1}^{2}+{v}p_{2}|=\frac{ 3{ D_{1}}^{2}}{4 {| D_{1}^{2}+4D_{2}|}+16 { D_{1}}}-\frac{ D_{1}}{12}.
\end{equation*}%
Therefore 
\begin{equation*}
|\gamma_{2}|-|\gamma_{1}| =|{u}p_{1}^{2}+{v}p_{2}|-|{s}p_{1}|=-\frac{ D_{1}}{12}-\frac{ {3 D_{1}}^{2}}{4 {| D_{1}^{2}+4D_{2}|}+16 { D_{1}}}.
\end{equation*}%
It completes the proof.
\end{proof}
	From above theorem, by using different subordination functions $\Phi$, we get the following result.
\begin{corollary}\label{cor1}
 Let $f$ be of the form \eqref{1.1}. Then for different classes we have the following bounds on $||\gamma_{2}|-|\gamma_{1}||$
	\footnotesize
	\renewcommand{\arraystretch}{2}	
	\begin{longtable}{|p{1.5cm}|p{1cm}|p{1.5cm}|p{5cm}|p{1.5cm}|}
		\hline
			\textbf{Authors} &\textbf{Classes} & \textbf{Upper Bounds} & \textbf{Lower Bounds}& \textbf{References} \\
			\hline
			\endfirsthead
			
			\multicolumn{4}{c}{{\tablename\ \thetable{} -- continued from previous page}} \\
			\hline
			Authors&Classes & Upper bounds & Lower bounds&references  \\
			\hline
			\endhead
			
			\hline
			\multicolumn{4}{r}{{Continued on next page}} \\
			\endfoot
			
			\hline
			\endlastfoot
			&$\mathcal{C}_{\alpha}$ & $\dfrac{\alpha}{6}$ & $\begin{cases}
			-\dfrac{1}{4}\,\alpha\,(2-\alpha), & \alpha\in\left(0, \dfrac{1}{3}\right],\\[2ex]
			-\dfrac{\alpha}{12}\left(\dfrac{6\alpha+13}{3\alpha+2}\right), &  \alpha\in\left(\dfrac{1}{3}, \dfrac{5}{6}\right),\\[2ex]
			-\dfrac{\alpha}{\sqrt{4+6\alpha}}, & \alpha \in \left[\dfrac{5}{6}, 1\right].
			\end{cases}$ &\\
			\hline
		&	$\mathcal{C}_{SG}$ & $\dfrac{1}{24}$ & $-\dfrac{23}{192}$  &\\
		\hline
		&	$\mathcal{C}_{qc}$& $\dfrac{c}{24}$ & $\dfrac{c^{2}}{192}-\dfrac{c}{8}$ &\\
		\hline
		&	$\mathcal{BS}(\alpha)$ & $\dfrac{1}{12}$ & $-\dfrac{11}{48}$  &\\
			\hline
		Mandal et al.	&	$\mathcal{C}_{l}$ & $\dfrac{1}{12}$ & $-\dfrac{4}{21}$  & \cite{gmj}\\
		\hline
			&	$\mathcal{C}_{c}$ & $\dfrac{1}{9}$ & $-\dfrac{49}{198}$ &\\
		\hline
			&	$\mathcal{C}_{\mathcal{RL}}$ & $-\dfrac{\sqrt{2}}{8}+\dfrac{5}{24}$ & $-\dfrac{305}{192}+\dfrac{17 \sqrt{2}}{16}$ &\\
		\hline
				&	$\mathcal{C}_\mathcal{L}$ & $\dfrac{1}{24}$ & $-\dfrac{23}{192}$ & \\
		\hline
				&	$\mathcal{C}_{s}$ & $\dfrac{1}{12}$ & $-\dfrac{11}{48}$  & \\
		\hline
			
		&	$\mathcal{C}_{e}$ & $\dfrac{1}{12}$ & $-\dfrac{4}{21}$  & \\
			\hline
		\end{longtable}
\end{corollary}

In following result, we establish the sharp bounds on the functional $||\Gamma_{2}|-|\Gamma_{1}||$.
\begin{theorem}\label{Theorem2.3}
	If $f\in\mathcal{C}(\Phi)$ is of the form \eqref{1.1} and $D_1>0$, then 
	\begin{equation*}
	|\Gamma_{2}|-|\Gamma_{1}| \leq 
	\begin{cases}
		\frac{{| 5D_{1}^{2}-4D_{2}|}-12{ D_{1}}}{48}, & \text{if }\ |5D_{1}^2-4D_{2}|\geq 16 D_{1}, \\ 
		\\
		\frac{D_{1}}{12}, & \text{if }\ |5D_{1}^2-4D_{2}|< 16 D_{1}.%
	\end{cases}%
	\end{equation*}%
	\newline
	\newline
	and 
	\begin{equation*}
	|\Gamma_{2}|-|\Gamma_{1}| \geq 
	\begin{cases}
		\frac{\left\vert 5D_{1}^2-4D_{2}\right\vert-12 D_{1} }{48}, & \text{if }\ 2 D_{1}\geq | 5D_{1}^{2}-4D_{2}|, \\[8pt]
		-\frac{D_{1}}{2}\sqrt{\frac{D_{1}}{4 D_{1}+|5D_{1}^2-4D_{2}|}}, & \text{if }\ 5 D_{1}\leq | 5D_{1}^{2}-4 D_{2}|, \\%
		[8pt]
		-\frac{ D_{1}}{12}-\frac{ {3 D_{1}}^{2}}{4 {| 5D_{1}^{2}-4D_{2}|}+16 { D_{1}}}, & \text{otherwise}.%
	\end{cases}%
	\end{equation*}%
	All of these bounds are sharp.
\end{theorem}

\begin{proof}
From \eqref{Gamma1Gamma2} and \eqref{2.4}, we have 
	\begin{equation*}
	\Gamma_{1}=-\frac{D_{1}}{8}p_{1},\quad \quad \Gamma_{2}:=\left( \frac{%
	5D_{1}^2+4D_{1}-4D_{2}}{192}\right) p_{1}^{2}-\frac{D_{1}}{24}p_{2}. 
	\end{equation*}%
	Therefore
	\begin{equation*}
	|\Gamma_{2}|-|\Gamma_{1}|=|{u}p_{1}^{2}+{v}p_{2}|-|sp_{1}|,
	\end{equation*}%
	where
	\begin{equation*}
	{s}:=\frac{D_{1}}{8},\quad \quad {u}:=\left( \frac{5D_{1}^2+4D_{1}-4D_{2}}{192}%
	\right) ,\quad \text{and}\quad {v}:=-\frac{D_{1}}{24}.
	\end{equation*}%
	By utilizing Lemma \ref{lemma1}, we have 
	\begin{align*}
	|\Gamma_{2}|-|\Gamma_{1}| & ={|4{u}t_{1}^{2}+{v}(2t_{1}^{2}+2(1-{|t_{1}|}^{2})t_{2})|}-2{|{s}t_{1}|} \\
	& \leq {j}{|t_{1}|}^{2}+2{v}\left( 1-{|t_{1}|}^{2}\right) |t_{2}|-2{|{s}t_{1}|},
	\end{align*}%
	with ${j}:=|4{u}+2{v}|$.\\
	As we know, $\mathcal{P}$ is rotationally invariant, so we can take, $t_{1}\in
	\lbrack 0,1]$ and by using the fact that $|t_{2}|\leq 1$, we get 
	\begin{align*}
	|\Gamma_{2}|-|\Gamma_{1}| & \leq \left( {j}-2{v}\right) t_{1}^{2}-2{s}t_{1}+2{v} 
	 =\Psi (x),
	\end{align*}%
	with 
	\begin{equation}
	\Psi (x)=k_{2}x^{2}+k_{1}x+k_{0},  \label{2.8}
	\end{equation}%
	where 
	\begin{equation*}
	k_{2}={j}-2{v},\quad \quad k_{1}=-2{s},\quad \quad k_{0}=2{v}.
	\end{equation*}%
	Now, there are two cases:\newline
	$\mathbf{A}_{1}:$ Let $k_{2}\leq 0$. As $k_{1}<0$, so from (\ref{2.8}), $\Psi ^{\prime }(x)=2k_{2}x+k_{1}<0$ and thus we have that $%
	\Psi $ is decreasing. Therefore 
	\begin{equation*}
	\Psi (x)\leq \Psi (0)=k_{0}=2{v}=\frac{D_{1}}{12}.
	\end{equation*}%
	$\mathbf{A}_{2}:$ If \(k_{2}>0\), then the function \(\Psi\) is quadratic and has leading coefficient positive . Hence, it is convex and attains its minimum value at the vertex, thus 
	\begin{equation*}
	\Psi (x)\leq \max \{\Psi (0),\Psi (1)\}.
	\end{equation*}%
	Here, if $k_{2}+k_{1}<0$, then $\Psi (1)=k_{2}+k_{1}+k_{0}<k_{0}=\Psi (0)$
	and it yields $\Psi (x)\leq \Psi (0)=2{v}$. Also, if $%
	k_{2}+k_{1}\geq 0$, then $\Psi (1)=k_{2}+k_{1}+k_{0}\geq k_{0}=\Psi
	(0)$. Therefore
	\begin{equation*}
	\Psi (x)\leq \Psi (1)=k_{2}+k_{1}+k_{0}=|4{u}+2{v}|-2{s}=\frac{{| 5D_{1}^{2}-4D_{2}|}-12{ D_{1}}}{48}.
	\end{equation*}%
	Thus, from the above discussion, we have 
		\begin{equation*}
		|\Gamma_{2}|-|\Gamma_{1}| \leq 
		\begin{cases}
			\frac{{| 5D_{1}^{2}-4D_{2}|}-12{ D_{1}}}{48}, & \text{if }\ |5D_{1}^2-4D_{2}|\geq 16 D_{1}, \\ 
			\\
			\frac{D_{1}}{12}, & \text{if }\ |5D_{1}^2-4D_{2}|< 16 D_{1}.%
		\end{cases}%
	\end{equation*}%
	For $|5D_{1}^2-4D_{2}|\geq 16 D_{1}$, the result is best possible for $f$ given by \eqref{2.2a} with $h(z)=(1+z)/(1-z)$. Also for $|5D_{1}^2-4D_{2}|< 16 D_{1}$ the result is sharp and the function $f$ given by \eqref{2.2a} with $h(z)=(1+z^{2})/(1-z^{2})$ provides the equality in the result.\newline
	Now, for the lower bound on $|\Gamma_{2}|-|\Gamma_{1}| $ by using Lemma \ref{lemma1}, we have 
	\begin{equation*}
	|\Gamma_{2}|-|\Gamma_{1}| =\left\vert {j}e^{i\theta }t_{1}^{2}+2{v}e^{i\phi
		_{1}}(1-t_{1}^{2})t_{2}\right\vert -2{s}t_{1},
	\end{equation*}%
	with
	\begin{equation*}
	{j}=\left\vert 4{u}+2{v}\right\vert ,\quad \quad \theta =arg(4{u}+2{v}),\quad \quad
	\phi _{1}=arg(2{v}).
	\end{equation*}%
	Since $\mathcal{P}$ is rotationally invariant, we may assume that \(t_{1}\in[0,1]\) and \(t_{2}=re^{i\phi_{2}}\), where \(r\in[0,1]\). Therefore
	\begin{equation*}
	|\Gamma_{2}|-|\Gamma_{1}| =\left\vert {j}e^{i\theta }t_{1}^{2}+2r{v}e^{i\phi
	}(1-t_{1}^{2})\right\vert -2{s}t_{1},
	\end{equation*}%
	with $\phi =\phi _{1}+\phi _{2}$. As $\left\vert e^{i\phi }\right\vert =1$,
	so 
	\begin{equation*}
	|\Gamma_{2}|-|\Gamma_{1}| =\left\vert {j}e^{i(\theta -\phi
		)}t_{1}^{2}+2r{v}(1-t_{1}^{2})\right\vert -2{s}t_{1}.
	\end{equation*}%
	Thus
	\begin{align}
	|\Gamma_{2}|-|\Gamma_{1}| =& \sqrt{{j}^{2}t_{1}^{4}+4{j}{v}rt_{1}^{2}(1-t_{1}^{2})\cos (\theta -\phi
		)+4{v}^{2}r^{2}(1-t_{1}^{2})^{2}}-2{s}t_{1},  \notag \\
	\geq & \left\vert {j}t_{1}^{2}-2r{v}(1-t_{1}^{2})\right\vert -2{s}t_{1},
	\label{2.9}
	\end{align}%
	as $\cos (\theta -\phi )\geq -1$. Let ${u}={v}=0$. Then 
	\begin{equation*}
	|\Gamma_{2}|-|\Gamma_{1}| \geq -2{s}t_{1}\geq -2{s}=-\frac{D_{1}}{4}.
	\end{equation*}%
	On the other hand, let ${u}\neq 0,\ {v}\neq 0$. Then from \eqref{2.9}, there are two cases:%
	\newline
	\textbf{B}$_{1}$: Let ${j}t_{1}^{2}-2r{v}(1-t_{1}^{2})\leq 0$. Then 
	\begin{equation*}
	t_{1}\leq \sqrt{\frac{2r{v}}{{j}+2r{v}}}:=\zeta _{1}.
	\end{equation*}%
	By using (\ref{2.9}) and the fact that, $r\in \lbrack 0,1]$, we have
	\begin{align*}
	|\Gamma_{2}|-|\Gamma_{1}| & \geq -({j}+2r{v})t_{1}^{2}-2{s}t_{1}+2r{v} \\
	& \geq -({j}+2r{v})\zeta _{1}^{2}-2{s}\zeta _{1}+2r{v} \\
	& =-2{s}\sqrt{\frac{2r{v}}{{j}+2r{v}}}=-2{s}\sqrt{\frac{r}{\frac{j}{2{v}}+r}}
	\\
	& \geq -2{s}\sqrt{\frac{1}{\frac{j}{2{v}}+1}} \\
	& ={-2{s}}\sqrt{\frac{2{v}}{{j}+2{v}}}\\
	& =-\frac{D_{1}}{2}\sqrt{\frac{D_{1}}{4 D_{1}+|5D_{1}^2-4D_{2}|}}.
	\end{align*}%
	\textbf{B}$_{2}$: If ${j}t_{1}^{2}-2r{v}(1-t_{1}^{2})>0$, then $t_{1}>\zeta
	_{1}$, and we consider a function $\varphi$ defined by
	\begin{equation*}
	\varphi (t_{1},r)=({j}+2r{v})t_{1}^{2}-2{s}t_{1}-2r{v}.
	\end{equation*}%
	Since, for ${s}^{2}\leq 2{v}({j}+2{v})$ and by using the fact that $r\in
	\lbrack 0,1]$, we get 
	\begin{equation*}
	\frac{\partial \varphi }{\partial t_{1}}=2({j}+2r{v})t_{1}-2{s}\geq 2\sqrt{%
		2r{v}({j}+2r{v})}-2{s}\geq 0.
	\end{equation*}%
	Therefore the function $\varphi$ is increasing, thus by using (\ref{2.9}), we obtain
	\begin{equation*}
	|\Gamma_{2}|-|\Gamma_{1}| \geq -2{s}\sqrt{\frac{2r{v}}{{j}+2r{v}}}\geq -2{s}\sqrt{\frac{2{%
				{v}}}{{j}+2{v}}}=-\frac{D_{1}}{2}\sqrt{\frac{D_{1}}{4 D_{1}+|5D_{1}^2-4D_{2}|}},
	\end{equation*}%
	as $r\in \lbrack 0,1]$. Now, for  ${s}\geq {j}+2{v}$, we have 
	\begin{equation*}
	\frac{\partial \varphi }{\partial t_{1}}=2({j}+2r{v})t_{1}-2{s}\leq 2{j}+4r{%
		{v}}-2{s}\leq 2{j}+4{v}-2{s}\leq 0.
	\end{equation*}%
	The function $\varphi $ is decreasing and from (\ref{2.9}), we get 
	\begin{equation*}
	|\Gamma_{2}|-|\Gamma_{1}|\geq {j}-2{s}=\frac{\left\vert 5D_{1}^2-4D_{2}\right\vert-12D_{1} }{48}.
	\end{equation*}%
	Lastly, for ${s}^{2}>2{v}({j}+2{v})$ and ${s}<{j}+2{v}$, we get
	\begin{equation*}
	\zeta _{1}<\frac{{s}}{{j}+2{v}}:=\zeta _{2}<1,
	\end{equation*}%
	and $\frac{\partial \varphi }{\partial r}=-2{v}(1-t_{1}^{2})\leq 0$ for $%
	r\in \lbrack 0,1]$ and $t_{1}\in \lbrack \zeta _{1},1]$. Thus
	\begin{equation*}
	\varphi (t_{1},r)\geq \varphi (t_{1},1)=({j}+2{v})t_{1}^{2}-2{s}t_{1}-2{{v}%
	}=\varphi_2(t_{1}).
	\end{equation*}%
	Since $\varphi_2^{\prime }(t_{1})=0$ yields $t_{1}=\zeta _{2}$. Therefore by using (\ref{2.9}), we deduce
	\begin{equation*}
	|\Gamma_{2}|-|\Gamma_{1}| \geq -2{{v}-\frac{{s}^{2}}{{j}+2{v}}}=-\frac{ D_{1}}{12}-\frac{ 3{ D_{1}}^{2}}{4 {| 5D_{1}^{2}-4D_{2}|}+16 { D_{1}}}.
	\end{equation*}%
	Now, all the above discussion implies
	\begin{equation*}
		|\Gamma_{2}|-|\Gamma_{1}| \geq 
		\begin{cases}
			\frac{\left\vert 5D_{1}^2-4D_{2}\right\vert-12 D_{1} }{48}, & \text{if }\ 2 D_{1}\geq | 5D_{1}^{2}-4D_{2}|, \\[8pt]
			-\frac{D_{1}}{2}\sqrt{\frac{D_{1}}{4 D_{1}+|5D_{1}^2-4D_{2}|}}, & \text{if }\ 5 D_{1}\leq | 5D_{1}^{2}-4D_{2}|, \\%
			[8pt]
			-\frac{ D_{1}}{12}-\frac{ {3 D_{1}}^{2}}{4 {| 5D_{1}^{2}-4D_{2}|}+16 { D_{1}}}, & \text{otherwise},%
		\end{cases}%
	\end{equation*}%
	For $ 2 D_{1}\geq | 5D_{1}^{2}-4D_{2}|$, it is a sharp result provided by the function $f$ given by \eqref{2.2a} with $
	h(z)=(1+z)/(1-z)$.\newline
	Let $5D_{1}\leq |5 D_{1}^{2}-4 D_{2}|$, then Lemma \ref{lemma1} ensures that the result is best possible for the function $f$ given by \eqref{2.2a}, with $h(z)$ is provided by \eqref{2.1} where
	\begin{equation*}
	t_{1}=\sqrt{\frac{{4 D_{1}}}{{| 5D_{1}^{2}-4D_{2}|}+4 D_{1}}},\quad \text{and}\quad t_{2}=%
	\begin{cases}
	\frac{ 5D_{1}^{2}-4D_{2} } {| 5D_{1}^{2}-4D_{2}|}, & \text{if }\  5D_{1}^{2}\neq 4D_{2}, \\[1.2em]
	1, & \text{if }\  5D_{1}^{2}= 4D_{2}.%
	\end{cases}%
	\end{equation*}%
	Let $ 5D_{1}^{2}= 4D_{2}$, the calculations are obvious and  $|\Gamma_{2}|-|\Gamma_{1}|=-\frac{D_{1}}{4}$.\newline
	Now, for $5D_{1}^{2}\neq 4D_{2}$,
	\begin{align*}
	|\Gamma_{2}|-|\Gamma_{1}| &=\left\vert {u}p_{1}^{2}+{v}p_{{2}}|-|{s}p_{1}\right\vert\\
	&=\frac{1}{48}\left\vert \left(5D_{1}^{2}-4D_{2}\right)t_{1}^{2}-4 D_{1} (1-t_{1}^{2})t_{2}\right\vert-\frac{1}{4} D_{1}t_{1}
	\end{align*}%
	and after some calculations, we deduce 
	\begin{equation*}
	|\Gamma_{2}|-|\Gamma_{1}|=-\frac{D_{1}}{2}\sqrt{\frac{D_{1}}{4D_{1}+|5D_{1}^2-4D_{2}|}}.
	\end{equation*}%
	Lastly, for $ 5D_{1}> | 5D_{1}^{2}-4 D_{2}|$ and $2D_{1}< | 5 D_{1}^{2}-4D_{2}|$, the result is best possible for the function $f$ defined by  \eqref{2.2a}, where $h(z)$ is given by \eqref{2.1} with 
	\begin{equation*}
	t_{1}=\frac{6 { D_{1}}}{ {| 5D_{1}^{2}-4D_{2}|}+4 { D_{1}}},\quad \text{and}\quad t_{2}=%
	\begin{cases}
	\frac{ 5D_{1}^{2}-4D_{2} }{| 5D_{1}^{2}-4D_{2}|}, & \text{if }\ D_{1}\neq 0, \\[1.2em]
	1, & \text{if }\ D_{1}=0.%
	\end{cases}%
	\end{equation*}%
	
	Now  
	\begin{align*}
	2304|{u}p_{1}^{2}+{v}p_{2}|^{2}& =\left\vert \left(5D_{1}^{2}-4D_{2}\right)t_{1}^{2}-4D_{1} (1-t_{1}^{2})t_{2}\right\vert ^{2}\\
	& =\left|5D_{1}^{2}-4D_{2}\right|^{2}t_{1}^{4}-2Re\left(4 D_{1}(\overline{5D_{1}^{2}-4D_{2}})t_{1}^{2}(1-t_{1}^{2})t_{2}%
	\right) +16D_{1}^{2}(1-t_{1}^{2})^{2} \\
	& =\left|5D_{1}^{2}-4D_{2}\right|^{2}t_{1}^{4}-2(4 D_{1}\left|5D_{1}^{2}-4D_{2} \right|t_{1}^{2}(1-t_{1}^{2})+16D_{1}^{2}(1-t_{1}^{2})^{2} \\
	& =\left( \left|5D_{1}^{2}-4D_{2}\right|t_{1}^{2}-4D_{1}(1-t_{1}^{2})\right) ^{2} \\
	& =\left(\frac{36{ D_{1}}^{2}}{ {| 5D_{1}^{2}-4D_{2}|}+4 { D_{1}}}-4D_{1}\right) ^{2}.
	\end{align*}%
	Since $5D_{1}> | 5D_{1}^{2}-4 D_{2}|$, it gives $\frac{36 { D_{1}}^{2}}{ {| 5D_{1}^{2}-4D_{2}|}+4 { D_{1}}}-4D_{1}>0$, therefore
	\begin{equation*}
	|{u}p_{1}^{2}+{v}p_{2}|=\frac{ 3{ D_{1}}^{2}}{4{| 5D_{1}^{2}-4D_{2}|}+16 { D_{1}}}-\frac{ D_{1}}{12}.
	\end{equation*}%
	It implies
	\begin{equation*}
	|\Gamma_{2}|-|\Gamma_{1}| =|{u}p_{1}^{2}+{v}p_{2}|-|{s}p_{1}|=-\frac{ D_{1}}{12}-\frac{ {3 D_{1}}^{2}}{4 {| 5D_{1}^{2}-4D_{2}|}+16 { D_{1}}}.
	\end{equation*}%
	It completes the proof.
\end{proof}
	From above theorem, by using different subordination functions $\Phi$, we get the following result.
\begin{corollary}\label{cor2}
	If $f\in\mathcal{C}(\Phi)$ is of the form \eqref{1.1}, then for different classes we have the following bounds on $||\Gamma_{2}|-|\Gamma_{1}||$
	\footnotesize
\renewcommand{\arraystretch}{2}
\begin{longtable}{|p{1.5cm}|p{1cm}|p{1.5cm}|p{5cm}|p{1.5cm}|}
	\hline
	\textbf{Authors} &\textbf{Classes} & \textbf{Upper Bounds} & \textbf{Lower Bounds}& \textbf{References} \\
	\hline
	\endfirsthead
	
	\multicolumn{4}{c}{{\tablename\ \thetable{} -- continued from previous page}} \\
	\hline
	Authors & Classes & Upper Bounds & Lower Bounds & References \\
	\hline
	\endhead
	
	\hline
	\multicolumn{4}{r}{{Continued on next page}} \\
	\endfoot
	
	\hline
	\endlastfoot
	Allu and Shaji	&$\mathcal{C}_{\alpha}$ & $\dfrac{\alpha}{6}$ & $\begin{cases}
	-\dfrac{1}{4}\,\alpha\,(2-\alpha), & \alpha\in\left(0, \dfrac{1}{3}\right],\\[2ex]
	-\dfrac{\alpha}{12}\left(\dfrac{6\alpha+13}{3\alpha+2}\right), &  \alpha\in\left(\dfrac{1}{3}, \dfrac{5}{6}\right),\\[2ex]
	-\dfrac{\alpha}{\sqrt{4+6\alpha}}, & \alpha \in \left[\dfrac{5}{6}, 1\right].
	\end{cases}$ & \cite{allu}\\
	\hline
	&$\mathcal{C}_{SG}$ & $\dfrac{1}{24}$ & $-\dfrac{31}{312}$ & \\
\hline
&	$\mathcal{C}_{qc}$& $\dfrac{c}{24}$ & $\left\{
\begin{array}{ll}
\dfrac{7 {| c|}^{2}}{192}-\dfrac{c}{8}, & c\in \left[0, \dfrac47\right] \\[3mm]
\dfrac{-7 c^{2}-26 c}{168 c+192}, & c\in \left(\dfrac47, 1\right] 
\end{array}
\right.$ & \\
\hline
&	$\mathcal{BS}(\alpha)$ & $\dfrac{1}{12}$ & $-\dfrac{1}{6}$  & \\
\hline
	Ahamed and Roy&$\mathcal{C}_{l}$ & $\dfrac{1}{12}$ & $-\dfrac{4}{21}$&  \cite{17600}   \\
\hline
&	$\mathcal{C}_{c}$ & $\dfrac{1}{9}$ & $-\dfrac{53}{234}$ & \\
\hline
&	$\mathcal{C}_{\mathcal{RL}}$ & $-\dfrac{\sqrt{2}}{8}+\dfrac{5}{24}$ & $\dfrac{-1239 \sqrt{2}+1771}{-2712+1728 \sqrt{2}}$ & \\
\hline
&	$\mathcal{C}_\mathcal{L}$ & $\dfrac{1}{24}$ & $-\dfrac{11}{120}$ &  \\
\hline
&	$\mathcal{C}_{s}$ & $\dfrac{1}{12}$ & $-\dfrac{1}{6}$ &  \\
\hline
&	$\mathcal{C}_{e}$ & $\dfrac{1}{12}$ & $-\dfrac{4}{21}$ &  \\
	\hline
\end{longtable}
\end{corollary}
\begin{remark}
From Theorem \ref{Theorem2.2} and Theorem \ref{Theorem2.3}, it is seen that the upper bounds on $|\gamma_{2}|-|\gamma_{1}|$ and $|\Gamma_{2}|-|\Gamma_{1}|$ are the same when 16 $D_1>\max\{|D_{1}^2+4 D_{2}|, |5D_{1}^2-4 D_{2}|\}$.\\
It is also clear from Corollary \ref{cor1} and Corollary \ref{cor2} that the following functions hold this property and have same upper bounds of $||\gamma_{2}|-|\gamma_{1}||$ and $||\Gamma_{2}|-|\Gamma_{1}||$.

	$\Phi_{{\alpha}}(z) =\left(\frac{1+z}{1-z}\right)^{\alpha}=1+2 \alpha  z+2 \alpha^{2} z^{2}+\cdots,$
	
	$\Phi_{{SG}}(z) =\frac{2}{1+e^{-z}}=1+\frac{1}{2} z -\frac{1}{24} z^{3}+\cdots,$
	
	$\Phi _{\mathcal{BS}}(z)=1+\frac{z}{1-\alpha z^{2}}=1+z+\alpha z^{3}+\cdots, $
	
	$\Phi _{l}(z)=z+\sqrt{1+z^{2}}=1+z+\frac{1}{2}z^{2}+\cdots ,$
	
	$\Phi _{c}(z)=1+\frac{4z}{3}+\frac{2z^{2}}{3},$
	
	$\Phi _{RL}(z)=\sqrt{2}-(\sqrt{2}-1)\sqrt{\frac{1-z}{1+2(\sqrt{2}-1)z}}=1+%
	\frac{5-3\sqrt{2}}{2}z+\frac{71-51\sqrt{2}}{8}z^{2}+\cdots ,$
	
	$\Phi _{L}(z)=\sqrt{1+z}=1+\frac{1}{2}z-\frac{1}{8}z^{2}+\cdots ,$
	
	$\Phi _{\sin }(z)=1+\sin (z)=1+z-\frac{1}{6}z^{3}+\cdots ,$
		
	$\Phi _{e}(z)=e^{z}=1+z+\frac{1}{2}z^{2}+\cdots.$

\end{remark}

\begin{remark}
	It can be observed from the Theorem \ref{Theorem2.2} and Theorem \ref{Theorem2.3} that the lower bounds on $|\gamma_{2}|-|\gamma_{1}|$ and $|\Gamma_{2}|-|\Gamma_{1}|$ are the same when 
	\[
	D_2 = \frac{D_1^2}{2}.
	\]
From Corollary \ref{cor1} and Corollary \ref{cor2}, we can see that the following functions hold this property and have same lower bounds of $||\gamma_{2}|-|\gamma_{1}||$ and $||\Gamma_{2}|-|\Gamma_{1}||$,\\
$\Phi_{{\alpha}}(z) =\left(\frac{1+z}{1-z}\right)^{\alpha}=1+2 \alpha  z+2 \alpha^{2} z^{2}+\cdots,$\\
$\Phi _{l}(z)=z+\sqrt{1+z^{2}}=1+z+\frac{1}{2}z^{2}+\cdots,$\\
$\Phi _{e}(z)=e^{z}=1+z+\frac{1}{2}z^{2}+\cdots .$\\		
\end{remark}

\end{document}